\documentclass[doublecolumn,letterpaper,10 pt]{ieeeconf} 
\usepackage{times} 
\usepackage{amsmath} 
\usepackage{amssymb,amsfonts}  
\usepackage{latexsym,theorem,epsfig}
\usepackage{graphicx} 
\usepackage{dsfont}
\usepackage{mathtools}
\usepackage{subfigure}
\usepackage{lscape}
\usepackage{color}
\usepackage{float}
\usepackage{enumerate}
\usepackage[usenames,dvipsnames]{pstricks}
\usepackage{epsfig}
\usepackage{pst-grad} 
\usepackage{pst-plot} 

\newtheorem{theorem}{Theorem}

{\theorembodyfont{\rmfamily} }
\newtheorem{lemma}[theorem]{Lemma}

\newtheorem{proposition}[theorem]{Proposition}

\newtheorem{rem}{Remark}

\allowdisplaybreaks

\title{Distributed Gaussian Learning over Time-varying Directed Graphs}

\author{Angelia Nedi\'{c}, Alex Olshevsky and C\'{e}sar A.\ Uribe
	\thanks{A. Nedi\'{c} (\textit{angelia.nedich@asu.edu}) is with the ECEE Department, Arizona State University. A. Olshevsky (\textit{alexols@bu.edu}) is with the ECE Department, Boston University. C.A. Uribe (\textit{cauribe2@illinois.edu}) is with the Coordinated Science Laboratory, University of Illinois. 
		This research is supported partially by the National Science Foundation under
		grants no.\ CNS 15-44953 and no.\ CMMI-1463262, and by the Office of Naval Research under grant
		no.\ N00014-12-1-0998.}  
}
\begin{document}
	\maketitle
	
	\begin{abstract}
		We present a distributed (non-Bayesian) learning algorithm for the problem of parameter estimation with Gaussian noise. The algorithm is expressed as explicit updates on the parameters of the Gaussian beliefs (i.e. means and precision). We show a convergence rate of $O(1/k)$ with the constant term depending on the number of agents and the topology of the network. Moreover, we show almost sure convergence to the optimal solution 
		of the estimation problem for the general case of time-varying directed graphs.
	\end{abstract}
	
	\section{Introduction}
	
	The analysis of distributed (non-Bayesian) learning algorithm gained popularity since the seminal work of 
	Jadbabaie et al.~\cite{jad12}. The ability of non-Bayesian updates to combine distributed optimization and learning algorithms make them especially useful for the design of distributed estimation algorithms with provable performance. 
	
	In the distributed learning setup, a group of agents repeatedly receive signals about a certain unknown state of the world or parameter. No single agent has enough information to accurately estimate the unknown state and, 
	thus, interaction with other agents is needed. 
	Several results are readily available for performance evaluation of distributed learning algorithms for a variety of scenarios. 
	Asymptotic exponential convergence rates where developed in~\cite{sha13,lal14b,sha16b}, 
	non-asymptotic bounds in~\cite{ned14}, time-varying directed graphs in~\cite{ned16}, conflicting hypotheses and linear rates in~\cite{ned15}, no-recall approaches to belief sharing in~\cite{rah15} and adversarial cases in~\cite{su16c,su16b}. 
	This list is necessarily incomplete, and the reader is referred to~\cite{ned16c} 
	for an extended set of references.
	
	Most of the previously proposed models assume that the parameter space of the estimation process is finite. Initial approaches to the study of continuum sets of hypotheses were developed in~\cite{nedic2016b}, where explicit non-asymptotic rates were derived. 
	A similar setup with Gaussian noisy observations with nonlinear function of the parameter to be estimated has been 
	considered in~\cite{sah16a,sah16b}, where almost sure convergence and asymptotic exponential rates for fixed undirected graphs were established.
	Allowing the hypotheses set to be infinite (e.g.\ a compact subset of $\mathbb{R}^m$)
	enables the exploration of traditional estimation problems in a distributed manner.
	One of such problems is the parameter estimation with Gaussian noise, which is the main concern of this manuscript. 
	
	In particular, we focus on the \textit{Gaussian case} of 
	the distributed (non-Bayesian) learning setup in~\cite{nedic2016b}. We analyze the belief update algorithm where agents observe a parameter corrupted by Gaussian noise and likelihood models are Gaussian functions, 
	which results in \textit{Gaussian beliefs}. 
	We present explicit updates for the beliefs' mean and variances, thus providing an algorithm
	for the distributed estimation process. We show almost sure convergence to an optimal parameter 
	and establish a convergence rate of $O(1/k )$. 
	We also provide simulation results for our algorithm and compare it with
	two approaches proposed in~\cite{chu16,bia16}.
	Our results hold for the general case of time-varying directed graphs, which are established 
	by using ideas from the push-sum algorithm in~\cite{ned16,ned13}.   
	
	This paper is organized as follows. Section~\ref{sec_distributed} describes the problem setup, as well as the proposed algorithm and main results. Section \ref{inde} provides a detailed comparison with results from \cite{chu16,bia16} for the case of identically distributed observations for all agents. Section~\ref{sec_simulations} shows simulation results and comparison with other algorithms. Finally, conclusions and future work are presented in Section~\ref{sec_conclusions}.
	
	\textbf{Notation}: superscripts refer to agents which are usually indexed by the letters $i$ or $j$.
	Subscripts indicate instants of time which are denoted by the letter $k$. Random variables are denoted by capital letters, e.g.\  $S_k^i$, and their corresponding realizations 
	by lower case letters, e.g.\ $s_k^i$.  
	The transpose of a vector $x$ is denoted as $x'$. 
	The term $\left[A\right]_{ij}$ denotes the entry of a matrix $A$ at the $i$-th row and the $j$-th column. 
	For a sequence $\{A_{k}\}$ of matrices we let $A_{k:t} = A_kA_{k-1}\cdots A_{t+1}A_t$ for $k \geq t$. 
	We denote the Gaussian function by
	\begin{align*}
	\mathcal{N}(\theta,\sigma^2) & 
	= \frac{1}{\sqrt{2\pi} \sigma^2} \exp \left(-\frac{(x -\theta)^2}{2\sigma^2}\right).
	\end{align*}
	
	\vspace{-0.8em}
	\section{Problem Setup, Algorithm and Results}\label{sec_distributed}
	
	Consider a group of $n$ agents whose goal is to collectively solve the following optimization problem
	\vspace{-0.5em}
	\begin{align}\label{opt_problem}
	\min_{\theta \in \Theta} F(\theta) & \triangleq
	\sum\limits_{i=1}^n D_{KL}\left(f^i\|\ell^i\left(\cdot|\theta\right)\right),
	\end{align}
	where $D_{KL}\left(f^i\|\ell^i\left(\cdot|\theta\right)\right)$ is 
	the Kullback-Leibler divergence between an \textit{unknown} distribution $f^i$ and a parametrized distribution $\ell^i( \cdot | \theta)$. Each agent $i$ has access to realizations of a random variable $S_k^i \sim f^i$ and a local family of 
	parametrized distributions $\{\ell^i(\cdot|\theta)\mid \theta \in \Theta\}$, where $\Theta$ is a set of parameters. 
	In other words, the agents want to determine
	a parameter $\theta^* \in \Theta$ corresponding to a distribution $\prod_{i=1}^{n} \ell^i(\cdot|\theta^*)$ 
	that is the closest to the distribution $\prod_{i=1}^{n}f^i$ in the sense of Kullback-Leibler divergence. 
	Moreover, the agents are allowed to interact over a sequence of time-varying directed graphs 
	$\{\mathcal{G}_k\}$, where $ \mathcal{G}_k=(\{1,\ldots,n\},E_k)$ and $E_k$ is a set of edges such that $(j,i) \in E_k$ indicates that agent $j$ can communicate with agent $i$ at time $k$.
	
	In~\cite{nedic2016b}, 
	the authors proposed an algorithm for solving the general problem of Eq.~\eqref{opt_problem} for compact sets $\Theta \subset \mathbb{R}^d$. 
	The algorithm generates non-Bayesian posteriors beliefs based on local observations 
	and shared beliefs from neighboring agents.
	Each agent $i$ constructs a sequence $\{\mu_k^i\}_{k=1}^{\infty}$ of beliefs about the hypothesis set $\Theta$,
	where $\mu_k^i$ maps measurable subsets of $\Theta$ to real values indicating 
	the belief that the unknown parameter $\theta^*$ is in the given subset.
	The algorithm proposed in~\cite{nedic2016b} is given by
	\begin{align}\label{protocol_cont}
	\bar{\mu}_{k+1}^i & \propto \prod\limits_{j=1}^{n} \left( \bar{\mu}_{k}^j\right)^{a_{ij}}  \ell^i(s_{k+1}^i|\cdot),
	\end{align}
	where $\bar{\mu}_k^i = d\mu_k^i / d\bar{\lambda}$ is a belief density function, see~\cite{sme05}, with respect to a reference measure $\bar{\lambda}$. Effectively, for a measurable subset $D \subseteq \Theta$, we have that the belief that $\theta^*$ is in $D$ is given by
	$\mu_k^i(D) = \int_{\theta \in D} \bar{\mu}_k^i \bar{\lambda}(\theta)$. Additionally, the scalar
	$a_{ij}$ is nonnegative and indicates how much agent $i$ weights the beliefs coming from its neighbor $j$, 
	with an understanding that $a_{ij}=0$ if no interaction between them occurs.
	
	In this manuscript, we assume that the observations have
	\textit{Gaussian distribution} and that the likelihoods models are Gaussian,
	both with bounded second order moments, 
	i.e.\ $S_k^i \sim \mathcal{N}(\theta^i,(\sigma^i)^2)$ and 
	$\ell^i(\cdot|\theta,\sigma^i) = \mathcal{N}(\theta,(\sigma^i)^2)$ where $\sigma^i > 0$ for every $i$. 
	This setting corresponds to the case of having measurements of the true parameter $\theta^*$ corrupted by 
	some Gaussian noise and the agents being informed that the noise is Gaussian with a known variance.
	
	The Kullback-Leibler distance between two univariate Gaussian distributions $p$ and $q$, 
	where $p = \mathcal{N}(\theta^1,(\sigma^1)^2)$ and $q = \mathcal{N}(\theta^2,(\sigma^2)^2)$ is given by
	\vspace{-0.3em}
	\begin{align*}
	D_{KL}\left(p\|q\right) & = \log \frac{\sigma^2}{\sigma^1} + \frac{(\sigma^1)^2 + (\theta^1 - \theta^2)^2}{(\sigma^2)^2} -\frac{1}{2}.
	\end{align*}
	Thus, in this case, the problem in Eq. \eqref{opt_problem} is equivalent to
	\vspace{-0.3em}
	\begin{align}\label{opt_problem2}
	\min_{\theta \in \Theta} \hat{F}(\theta) & \triangleq
	\sum\limits_{i=1}^n \frac{(\theta - \theta^i)^2}{2 (\sigma^i)^2}, 
	\end{align}
	which is convex with a unique solution 
	\vspace{-0.3em}
	\begin{align}\label{opt_theta}
	\theta^* = \sum\limits_{i=1}^{n}\frac{\theta^i / (\sigma^i)^2}{\sum\limits_{j=1}^{n}1 /(\sigma^j)^2}.
	\end{align}
	
	However, the exact value of $\theta^i$ is unknown and each agent $i$ has access only to noisy observations of the form 
	${S_k^i = \theta^i + \epsilon^i}$, where $\epsilon^i \sim \mathcal{N}(0,(\sigma^i)^2)$. Moreover, variances are only known locally, i.e.\ agent $i$ only knows $\sigma^i$. 
	
	We propose the following distributed algorithm for solving the problem in Eq.~\eqref{opt_problem2} over time-varying directed graphs
	\begin{subequations}\label{ours}
		\begin{align}
		\tau_{k+1}^i & =  \sum\limits_{j=1}^{n}\left[A_k\right]_{ij}\tau_k^j + \tau^i \label{ours_a}\\
		\theta_{k+1}^i & = \frac{\sum\limits_{j=1}^{n} \left[A_k\right]_{ij}\tau_k^j \theta_k^j + s_{k+1}^i\tau^i }{\tau_{k+1}^i} \label{ours_b}
		\end{align}
	\end{subequations}
	where $\tau^i = 1 / (\sigma^i)^2$ is refereed as the precision of the observations. 
	The weights $\left[A_k\right]_{ij}$ are chosen as
	\begin{align} \label{matrix}
	\left[A_k\right]_{ij} & = \begin{cases}
	\frac{1}{d_k^j+1} & \text{if } (j,i) \in E_k,\\
	0 & \text{otherwise}
	\end{cases}
	\end{align}
	where $d_k^j$ is the out-degree of node $j$ at time $k$. 
	Without loss of generality, we assume that $\tau_0^i=\tau^i$ for all~$i$.
	
	\begin{rem}
		It is {\upshape not} necessary for each agent to have some form of informative observations. Indeed, there might be agents with no observations working as buffers for information for which we also expect correct estimates of $\theta^*$. These ``blind" agents depend on communicating with other agents to construct its estimates. 
	\end{rem}
	
	\begin{rem}
		While our focus is in on the univariate Gaussian case,
		extensions to the multivariate are similarly possible using the results of conjugate priors for multivariate Gaussian distributions.
	\end{rem}
	
	The next proposition shows that the algorithm in Eq.~\eqref{ours} is 
	a specific realization of Eq.~\eqref{protocol_cont} 
	for the case of Gaussian distributions in the priors and likelihood models.
	
	\begin{proposition}\label{prop1}
		Let the prior belief density $\bar{\mu}_0^i$ of every agent  be a Gaussian function, i.e. 
		\begin{align*}
		\bar{\mu}_0^i(\theta;\theta_0^i,\sigma^i) & =\mathcal{N}(\theta_0^i,(\sigma^i)^2)
		\end{align*}
		and let the parametric family of distributions for the likelihood models be Gaussian functions, i.e.
		\begin{align*}
		\ell^i(s|\theta;(\sigma^i)^2) & =\mathcal{N}(\theta,(\sigma^i)^2).
		\end{align*}
		Then, for any $k \geq 1$, the posterior belief density $\bar{\mu}_{k}^i$, given by Eq.~\eqref{protocol_cont}, 
		is also a Gaussian function. Moreover, if the weights $a_{ij}$ are chosen to be $1 / (d_k^j+1)$, then the mean and the standard deviation of the posterior follow Eq.~\eqref{ours}.
	\end{proposition}
	
	Before presenting our main results, we state two auxiliary lemmas from~\cite{ned13} that 
	describe the geometric convergence for the product of column stochastic matrices.
	\begin{lemma}\label{lemma_angelia}
		[Corollary 2.a in \cite{ned13}] Let the graph sequence $\{\mathcal{G}_k\}$
		be B-strongly connected\footnote{There is an integer $B\ge 1$ such that the graph $\left(V,\bigcup_{i=kB}^{\left(k+1\right)B-1}E_i\right)$ is strongly connected for all $k \geq 0$}. Then, there is a sequence $\{\phi_k\}$ of stochastic vectors such that
		\begin{align*}
		|\left[A_{k:t}\right]_{ij} - \phi_k^i| & \leq C \lambda^{k-t}  \ \ \ \ \ \ \text{for all } \ k \geq t \geq 0
		\end{align*}
		where $\{A_k\}$ is as in Eq.~\eqref{matrix} and 
		the constants $C$ and $\lambda$ satisfy the following relations:\\
		\noindent (1)
		For general $B$-strongly connected
		graph sequences $\{\mathcal{G}_k\}$ 
		\begin{align*}
		C = 4, & \ \ \ \ \lambda = \left(1-\frac{1}{n^{nB}}\right)^{\frac{1}{B}}. \ 
		\end{align*}
		\noindent (2) If every graph $\mathcal{G}_k$ is regular with $B=1$
		\begin{align*}
		C = \sqrt{2}, & \ \ \ \ \lambda = 1-1/4n^3 .
		\end{align*}
	\end{lemma}
	
	\begin{lemma} \label{lemma_deltabound}
		[Corollary 2.b in \cite{ned13}] 
		Let the graph sequence $\left\{\mathcal{G}_k\right\}$ be B-strongly connected, and define
		\begin{align*}
		\delta \triangleq\inf_{k\geq 0} \left(\min_{1\leq i\leq n}\left[A_{k:0} \mathbf{1}_n \right]_i\right).
		\end{align*}
		Then, $\delta \geq 1/n^{nB}$. Moreover, if every $\mathcal{G}_k$ is regular and strongly connected (i.e. $B=1$), then 
		$\delta=1$.
		Furthermore, the sequence
		$\{\phi_k\}$ from Lemma \ref{lemma_angelia} satisfies $\phi_k^j \geq \delta/n$ 
		for all $k \geq 0$ and $j = 1, \ldots, n$.
	\end{lemma}
	
	Now, we proceed to state our two main results 
	showing the convergence properties of the algorithm in Eq.~\eqref{ours}.
	\begin{lemma}\label{lemma_main}
		The expected mean process $\{\mathbb{E}[\theta_k^i]\}$ converges to 
		$\theta^*$ for all $i$ with a convergence rate of $O(1 / k)$. Moreover, the constant terms depend on the topology of the network, the precision of the observations and the initial guess.
	\end{lemma}
	
	\begin{proof}
		In fact, we will prove the bound
		\begin{align}\label{main_lemma}
		\left| \mathbb{E}[ \theta_{k+1}^i] - \theta^*\right| \leq \frac{\tau_{max}  }{  \tau_{min} k \delta} \left(\|  \theta_0 -\theta^* \boldsymbol{1} \|_1 + \frac{2 C\|\theta - \theta^*\boldsymbol{1} \|_1 }{1- \lambda} \right)  
		\end{align}
		with $\tau_{max} = \max_j \tau^j$, and $\tau_{min}$ is the smallest non-zero precision among all agents.
		
		First, define a new variable as $x^i_k  = \tau^i_k\theta^i_k$, then from Eq.~\eqref{ours_b} it follows that
		\begin{align*}
		x_{k+1} & = A_kx_k + \text{diag}(\tau) s_{k+1} \\
		& = A_{k:0}x_0 + \sum_{t=1}^{k}A_{k:t}\text{diag}(\tau) s_t + \text{diag}(\tau) s_{k+1} 
		\end{align*}					
		where $\text{diag}(\tau)$ is a diagonal matrix with ${\left[\text{diag}(\tau)\right]_{ii} =\tau^i}$ and ${x_{k} = [x_k^1, \ldots, x_k^n]'}$, 	${\tau = [\tau^1, \ldots, \tau^n]'}$, ${s_{k} = [s_k^1, \ldots, s_k^n]'}$.
		\\
		Adding and subtracting $\sum_{t=1}^{k}\phi_k \tau' s_t$ from the preceding relation we obtain
		{\small
			\begin{align*}	
			x_{k+1} & = A_{k:0}x_0 + \sum_{t=1}^{k}D_{k:t}\text{diag}(\tau)s_t + \text{diag}(\tau)s_{k+1} + \sum_{t=1}^{k} \phi_k \tau' s_t 
			\end{align*}}
		with $D_{k:t} = A_{k:t} - \phi_k \boldsymbol{1}'$, and $\phi_k$ is as in Lemma~\ref{lemma_angelia}. 
		
		Following a similar procedure, from Eq.~\eqref{ours_a} it holds that
		\begin{align*}
		\tau_{k+1} & = A_{k:0}\tau_0 + \sum_{t=1}^{k} D_{k:t}\tau +k\phi_k \boldsymbol{1}'\tau +\tau.
		\end{align*}
		
		Going back to the original variable $\theta_k$, we have that
		{
			\begin{align*}	
			&\mathbb{E}[\theta_{k+1}^i]  = \\ & \frac{[A_{k:0}\text{diag}(\tau)\theta_0]_i + \sum_{t=1}^{k}[D_{k:t}\text{diag}(\tau)\theta]_i + \tau^i\theta^i + k \phi_k^i \tau' \theta}{[A_{k:0}\tau_0]_i + \sum_{t=1}^{k} [D_{k:t}\tau]_i +k\phi_k^i \boldsymbol{1}'\tau +\tau^i}
			\end{align*}}
		By subtracting $\theta^*$ on both sides of the previous relation and taking the absolute value, we obtain
		\begin{align*}
		& \left| \mathbb{E}[ \theta_{k+1}^i] - \theta^*\right| 
		\leq \left|  \frac{   \left[A_{k:0} \text{diag}(\tau_0)\left( \theta_0 -\theta^* \boldsymbol{1}\right) \right]_i }{\sum_{t=1}^{k} [D_{k:t}\tau]_i + k \phi_k^i \boldsymbol{1}'\tau} \right| 
		+
		\\
		& \left|  \frac{  \tau^i\left( \theta^i -\theta^*\right)   }{\sum_{t=1}^{k} [D_{k:t}\tau]_i + k \phi_k^i \boldsymbol{1}'\tau} \right| +
		\left|  \frac{   \sum_{t=1}^{k} \left[  D_{k:t}\text{diag}(\tau)\left( \theta -\theta^* \boldsymbol{1}\right)  \right]_i   }{\sum_{t=1}^{k} [D_{k:t}\tau]_i + k \phi_k^i \boldsymbol{1}'\tau} \right|
		\end{align*}
		where the terms involving $k \phi_k^i \tau' \theta$ cancel out and the following positive terms are removed from the denominator $\left[ A_{k:0}\tau_0\right]_i  +\tau^i > 0$. 
		
		Then by the fact that $[D_{k:t}\boldsymbol{1}]_i + \phi_k^i n > \delta$ on the denominator and using Lemma \ref{lemma_angelia} on the third term it follows that
		\begin{align*}
		&\left| \mathbb{E}[ \theta_{k+1}^i] - \theta^*\right|  			  \leq \left|  \frac{   \left[A_{k:0} \text{diag}(\tau_0)\left( \theta_0 -\theta^* \boldsymbol{1}\right) \right]_i }{ k \delta \tau_{min}} \right| 
		+
		\\
		& \qquad \frac{  \tau^i |\theta^i -\theta^*|  }{ k \delta \tau_{min}  }  +
		\frac{C \tau_{max} \| \theta -\theta^*\boldsymbol{1} \|_1 }{k \delta \tau_{min} (1-\lambda)} .
		\end{align*}
		
		Finally, the desired result follows by H\"olders inequality in the first term with $\|[A_{k:0} \text{diag}(\tau)]_i\|_{\infty} = \tau_{max}$ and grouping the second and third terms since $\frac{C}{1-\lambda} > 1 $.
		\begin{align*}
		&\left| \mathbb{E}[ \theta_{k+1}^i] - \theta^*\right|  			  \leq  \frac{   \max_j [A_{k:0}]_{ij} \tau^j\| \theta_0 -\theta^* \boldsymbol{1}\|_1  }{ k \delta \tau_{min}} 
		+
		\\
		& \qquad +
		\frac{2 C \tau_{max} \| \theta -\theta^*\boldsymbol{1} \|_1 }{k \delta \tau_{min} (1-\lambda)} .
		\end{align*}
	\end{proof}
	
	The first term in Eq. \eqref{main_lemma} shows the dependency on the initial estimates $\theta_0$ while the second term shows depends on the heterogeneity of mean of local observations. The network topology and the number of agents is characterized by $\lambda$ and $\delta$. 
	
	We are now ready to state our main result about the almost sure convergence of the proposed algorithm.
	
	\vspace{-1.0em}
	\begin{theorem}\label{almost}
		Let the graph sequence of interactions $\{\mathcal{G}_k\}_{k=1}^{\infty}$ be B-strongly connected. Moreover, assume ${S_k^i \sim \mathcal{N}(\theta^i,(\sigma^i)^2)}$ and $\ell^i(\cdot|\theta) = \mathcal{N}(\theta,(\sigma^i)^2)$ for all $i$. Then, the sequence $\{\theta_k^i\}$ generated by Eq. \eqref{ours} converges almost surely to $\theta^*$, i.e.
		\begin{align*}
		\lim_{k\to \infty} \theta_k^i & = \theta^* \ \ \ \ \ \text{a.s.} \ \ \forall i 
		\end{align*}
	\end{theorem}
	
	A proof of Theorem \ref{almost} is not shown due to space constraints. Nonetheless, its result follows by the bounded variance assumption of the observations and the weighted law of large numbers in \cite{pru66}.
	\vspace{-1.0em}
	\begin{rem}
		The specific selection of weights as $1 / (d_k^j+1)$ is a design choice. Theorem \ref{almost} still holds for any sequence of column stochastic matrices $\{A_k\}$ with every non-zero entry bounded from bellow away from zero, and with positive diagonal entries. 
	\end{rem}
	\vspace{-2.0em}
	\section{Identical distributions for all agents}\label{inde}
	
	A specific version of the proposed problem is the case when all agents observe independent realizations of the same random variable, i.e. $S_k^i \sim \mathcal{N}(\theta^*,(\sigma^2)^*)$. Recently, authors in \cite{chu16,bia16} have explored this case. Specifically, in \cite{bia16} the authors are concerned with the effects of the network topology on the convergence rate of the distributed mean estimation problem.  They show mean square consistency of the following algorithm
	\begin{align}\label{biau}
	\theta^i_{k+1} & = \frac{k}{k+1}\sum_{j=1}^{n}a_{ij} \theta_k^j +\frac{1}{k+1}s_{k+1}^i,
	\end{align}
	and provide explicit rates for different network topologies. Note that the algorithm in Eq. \eqref{biau} reduces to Eq. \eqref{ours} when $\tau^i = 1$ in such a way that $\tau_k^i = k$ for all $i$, and the graph is static with a doubly stochastic weight matrix.
	
	In \cite{chu16}, the authors proposed a new distributed Gaussian learning algorithm where communication between agents is noisy. Following the non-Bayesian learning without recall approach proposed in \cite{rah15} they develop the specific realization for Gaussian random variables. Additionally, they consider the sequence of observations $\{s_k^i\}$ as coming from an agent, denoted as $n+1$, and thus a different weighting strategy is proposed. Their algorithm is
	\begin{subequations}\label{chazelle}
		\begin{align}
		\tau_{k+1}^i & = \tau_k^i + d_k^j \tau \\
		\theta_{k+1}^i & = \frac{\sum_{j=1 }^{n+1}\tau_k^j a_k^j}{\tau_{k+1}^i}
		\end{align}
	\end{subequations}
	with the specific condition that $\tau_k^j = \tau$ for all $j \neq i$, $a_k^j =\theta_k^i$ for $j=i$ and $a_k^j = \theta_k^j +\epsilon$ with $\epsilon \sim \mathcal{N}(0,\tau)$, with $a_{k}^{n+1} = s_k^i$. The authors showed almost sure convergence of the algorithm. Moreover, a convergence rate of $O(k^{-\frac{\gamma}{2d}})$ was derived, where $\gamma$ is a bound on the uniform connectivity to the truth observations and $d$ is the maximal degree over all the networks.
	
	One particular characteristic of the algorithm proposed in \cite{chu16} is that, apart from traditional literature on distributed learning, the authors do not assume agents communicate over a \textit{ sufficiently }connected network ($B$-strong connectivity in Theorem \ref{almost}). They replace this assumption by a so-called \textit{truth-hearing assumption} which works as a $1 /\gamma$-strong connectivity with the $n+1$ node that provides direct noisy observations of $\theta^*$. Thus, it is required that every node receives signals from node $n+1$ at least once in every time interval of length $1 / \gamma$. If all agents receive independent observations from identical distributions, connectivity of the network and truth hearing assumptions both serve the same purpose of guarantying the diffusion of the information over the network, otherwise some form of connectivity between agents is needed.
	
	In addition to different connectivity assumptions, one main characteristic of the algorithm in Eq. \eqref{chazelle} is that agents do not differentiate the signal $S_k^i$ coming from the observations of the parameter, and the signals $\{a_k^j\}$ coming from other agents. Every agent treats both signals similarly. The weights for observations of $S_k^i$ and neighbors signals $\{\theta_k^i\}_{i=1}^n$ decay. Whereas in our approach in Eq. \eqref{ours} the weight for $S_k^i$ decays to zero and the weight for the convex combination of $\{\theta_k^i\}_{i=1}^n$ goes to one. This indeed shows that we do require the identification of signals coming from either agents or the noisy parameter observations. This extra information could explain why our approach has better performance in terms of convergence rates. 
	
	\section{Simulations}\label{sec_simulations}
	
	In this section, we provide simulation results for our proposed algorithm and we compare its performance with results in \cite{chu16,bia16}. Initially, we will consider the same scenario as in \cite{chu16,bia16} with static undirected graphs with all agents having identical distributions in their \textit{noiseless beliefs sharing}. We will evaluate the performance of the algorithms for two different graphs topologies, namely: path/line graph and a lattice/grid graph. 
	
	\vspace{-1em}
	\begin{figure}[ht]
		\centering
		\includegraphics[width=0.5\textwidth]{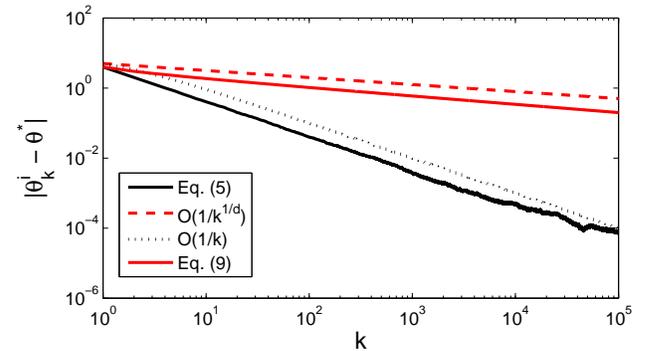}
		\caption{Simulations results of algorithms in Eq. \eqref{ours} and Eq. \eqref{chazelle} for a lattice/grid graph of 25 nodes for an average behavior over 500 Monte Carlo simulations.}
		\label{grid25}
	\end{figure}
	
	\vspace{-0.5em}
	Figure \ref{grid25} shows the absolute error of the estimated value $\theta^*$ for the lattice/grid graph with 25 agents. It is assumed that $S_k^i \sim \mathcal{N}(4,1)$. An average over 500 Monte Carlo simulations is shown for one arbitrary agent. In addition, the theoretical convergence rates are also shown for comparison purposes. No simulation of the algorithm in Eq. \eqref{biau} is shown since it reduces to the same algorithm as in Eq. \eqref{ours} for the simulated scenario.
	
	Figure \ref{path25} shows the simulation results for the same scenario as in Figure \ref{grid25} but now for a path/line graph of 15 agents. As predicted by the theoretical convergence rate bounds, the proposed algorithm in Eq. \eqref{ours} decays as $O(1/k)$ where the topology of the network affects only the constant whereas the proposal in Eq. \eqref{chazelle} depends explicitly  on the maximum degree among all graphs as $O(1/k^{1/d})$.  
	
	\vspace{-1em}
	\begin{figure}[ht]
		\centering
		\includegraphics[width=0.5\textwidth]{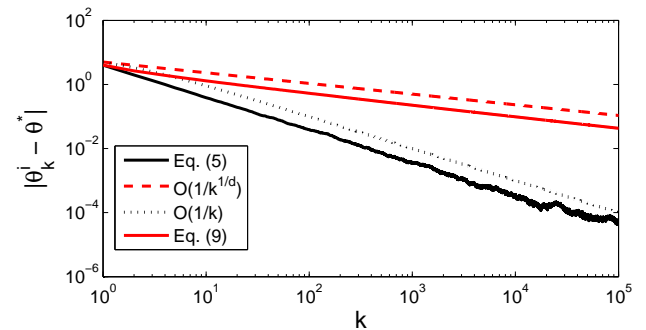}
		\caption{Simulations results of algorithms in Eq. \eqref{ours} and Eq. \eqref{chazelle} for a path graph of 25 nodes. Average behavior over 500 Monte Carlo simulations.}
		\label{path25}
	\end{figure}
	
	\vspace{-0.5em}
	Next, we will show that for the case of each agent having noise with different standard deviations, by using information about the current estimate precision (i.e. $\tau_k^i$) a better performance is achieved. Figure \ref{d_tau} shows the absolute error on the estimation of $\theta^*$ for the algorithm in Eq. \eqref{ours} that uses precision information and the proposal in Eq. \eqref{biau} that assumes uniform precision. In this simulation, agents have heterogeneous precisions such that $S_k^i \sim \mathcal{N}(4,i)$. That is, in the path graph, the first agent has $\tau^1 = 1$, the last agent, on the other hand, has $\tau^n = n$. This implies that agent $1$ has the highest variance in its observations. We have chosen to show the results for agent $1$ only.    
	
	\vspace{-1em}
	\begin{figure}[ht]
		\centering
		\includegraphics[width=0.5\textwidth]{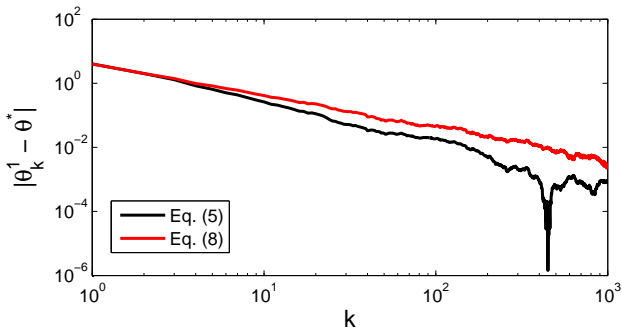}
		\caption{Simulations results of algorithms in Eq. \eqref{ours} and Eq. \eqref{biau} for a path graph of 25 nodes with heterogeneous precisions (i.e. $\tau's$). Average behavior over 500 Monte Carlo simulations.}
		\label{d_tau}
	\end{figure}
	
	\vspace{-0.5em}
	Finally, we will present the simulation results for a directed static graph which has been shown to be a pathological case for the push-sum algorithm, see Figure \ref{bad}. Each agent receives signals of the form $S_k^i \sim \mathcal{N}(i,n-i+1)$. Thus every agent has different measurement precisions and different $\theta^i$. The optimal $\theta^*$ as defined in Eq. \eqref{opt_theta}.
	
	\begin{figure}[ht]
		\centering
		\includegraphics[width=0.4\textwidth]{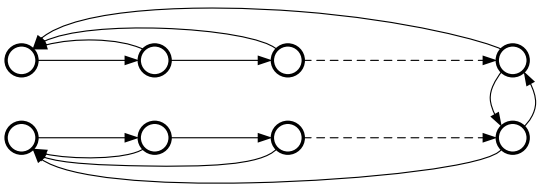}
		\caption{Directed graph for simulation of the algorithm in Eq. \eqref{ours}.}
		\label{bad}
	\end{figure}
	
	\vspace{-0.5em}
	Figure \ref{bad_result} shows the simulation results for the algorithm in Eq. \eqref{ours} to the specific set of observations ${S_k^i \sim \mathcal{N}(i,n-i+1)}$ on the graph in Figure \ref{bad}. The average over $10$ Monte Carlo simulation is shown. The predicted $O(1/k)$ behavior is observed, after a transition time that depends on the number of agents in the network, (i.e. the effects on $n$ and $\lambda$ in Lemma \ref{lemma_main}).   
	
	\begin{figure}[ht]
		\centering
		\includegraphics[width=0.5\textwidth]{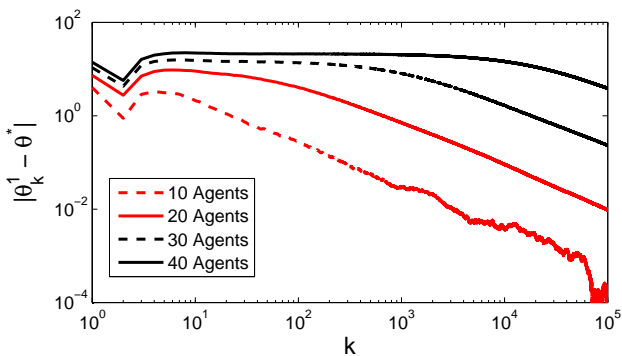}
		\caption{Simulations results of algorithms in Eq. \eqref{ours} for the graph depicted in Figure \ref{bad}. Four different results are shown, for $10$, $20$, $30$ and $40$ agents respectively.}
		\label{bad_result}
	\end{figure}
	\section{Conclusions}\label{sec_conclusions}
	
	We developed an algorithm for distributed parameter estimation with Gaussian noise over time-varying directed graphs. The proposed algorithm is shown to be a specific case of a more general class of distributed (non-Bayesian) learning methods. Almost sure converge as well as an explicit convergence rate is shown in terms of the network topology and the number of agents. Comparisons with recently proposed approaches are presented.
	Future work should consider nonlinear observations of the parameter $\theta$, that is $S_k^i \sim \mathcal{N}(g^i(\theta),(\sigma^i)^2)$ for some function ${g: \Theta \to \mathbb{R}}$. Ongoing work develops similar parameter estimation approaches for the larger case of the exponential family of distributions on the natural parameter space. A particularly interesting case is when the parameter $\theta^*$ is changing with time, either arbitrarily, on some form of Markov process or other dependencies. This case renders observations to be not identically distributed nor independent.
	
	\bibliographystyle{IEEEtran} 
	
	\bibliography{IEEEfull,bayes_cons_3}
	
\end{document}